\documentclass[11 pt]{amsart}
\usepackage{amssymb} 
\usepackage{amsmath,amsthm,amscd,graphicx}
\usepackage{color}
\usepackage{hyperref}
\hypersetup{ 
colorlinks=true, %colorise les liens 
breaklinks=true, %permet le retour ? la ligne dans les liens trop longs 
urlcolor= blue, %couleur des hyperliens 
linkcolor= blue, %couleur des liens internes 
bookmarksopen=true, %si les signets Acrobat sont cr??s, 
bookmarksnumbered=true,
pdftitle={A topological obstruction to the controllability of  nonlinear  equations  with bilinear control term}, %informations apparaissant dans 
pdfauthor={Thomas Chambrion and Laurent Thomann, IECL, Universit\'e de Lorraine, France}, %dans les informations du document 
pdfsubject={Proof of meagerness in Baire sense of the attainable set  of some nonlinear partial differential equations  with a bounded bilinear control term},
pdfkeywords={Control theory, bilinear control, obstructions, nonlinear Klein-Gordon equation} %sous Acrobat. 
} 

\usepackage{enumerate}
\usepackage{pdfsync}
\newtheorem{thm}{Theorem}[section]
 
 \newtheorem{lem}[thm]{Lemma}
 
 \newtheorem{proposition}[thm]{Proposition}
 
 \theoremstyle{definition}
 \newtheorem{definition}[thm]{Definition}
 \newtheorem{assumption}[thm]{Assumption}
 \newtheorem{rem}[thm]{Remark}
 
 \numberwithin{equation}{section}

\def\be#1 {\begin{equation} \label{#1}}
\newcommand{\ee}{\end{equation}}
\renewcommand{\phi}{\varphi}

\def\R{\mathbf R}

\def\M{\mathcal M}
\def\X{\mathcal X}

\def\e{e}

\def\dis{\displaystyle}    
\newcommand{\1}{{\bf 1}}

\definecolor{gr}{rgb}   {0.,   0.69,   0.23 }
\definecolor{bl}{rgb}   {0.,   0.5,   1. }
\definecolor{mg}{rgb}   {0.85,  0.,    0.85}
\definecolor{yl}{rgb}   {0.8,  0.7,   0.}
\definecolor{or}{rgb}  {0.7,0.2,0.2}

\catcode`\ = \active\def {\'e} \catcode`\ = \active\def {\'E}  \catcode`\Ë = \active\def Ë{\`A}\catcode`\ = \active\def {\`e}
\catcode`\ = \active\def {\c{c}} \catcode`\ = \active\def
{\`a} \catcode`\ = \active \def {\^e} \catcode`\ = \active
\def {\`u} \catcode`\ = \active \def {\^o} \catcode`\ =
\active \def {\^u} \catcode`\ = \active \def {\^{\i}}
\catcode`\ = \active \def {\^a} \catcode`\ = \active \def
{{\"o}}

\begin{document}

\title[Topological obstruction to controllability]{A topological obstruction to the controllability of  nonlinear wave equations  with bilinear control term}

\author{Thomas Chambrion and Laurent Thomann}



\newcommand{\Addresses}{{ 
  \bigskip
  \footnotesize

  Thomas Chambrion, \textsc{Universit\'e de Lorraine, CNRS, INRIA, IECL, F-54000 Nancy, France}\par\nopagebreak
  \textit{E-mail address}:   \texttt{Thomas.Chambrion@univ-lorraine.fr}

  \medskip

  Laurent Thomann, \textsc{Universit\'e de Lorraine, CNRS, IECL, F-54000 Nancy, France}\par\nopagebreak
  \textit{E-mail address}:   \texttt{Laurent.Thomann@univ-lorraine.fr}

  \medskip

}}

 \begin{abstract}
 In this paper we prove that the  Ball-Marsden-Slemrod controllability obstruction also holds for nonlinear equations, with $L^1$ bilinear controls. We first show an abstract result and then we apply it to nonlinear wave equations. The first application to the Sine-Gordon equation directly follows from the abstract result, and the second application concerns the cubic wave/Klein-Gordon equation and needs some additional work.
\end{abstract}

\maketitle

\tableofcontents

 \newpage

\section{Introduction and main result}

\subsection{Introduction}
Evolution equations with a bilinear control term are often used to model the 
dynamics of a system driven by an external field (for instance, a quantum system 
driven by an electric field). In view of their importance, very few satisfactory 
descriptions of the attainable sets of such systems are available (among the rare 
exceptions, see  Beauchard \cite{Beauchard05} for the case of the linear 
Schr\"odinger equation on a 1D compact domain or \cite{Beauchard11} for the linear wave equation on a 1D compact domain). For an overview of controllability results of bilinear control systems, we refer to Khapalov~\cite{Khapalov}. \medskip

Roughly speaking, the attainable set for such systems does not coincide with the 
natural functional space where the system is defined. An explanation    was provided by 
a celebrated article by Ball, Marsden and 
Slemrod  \cite{BMS} who proved that the attainable 
set of the linear dynamics with a bounded bilinear control using $L^r$, $r>1$ real 
valued controls,  is contained in a countable union of compact sets. This result 
has been adapted to the case of the Schr\"odinger equation by 
Turinici \cite{Turinici}. For 
partial differential equations posed in an infinite dimensional Banach space,  
this represents a strong topological obstruction to  controllability (since 
the attainable set has hence empty interior by the Baire theorem). The proof 
heavily relies on the reflectiveness of $L^r$, $r>1$ and could not be directly 
extended to $L^1$ controls. \medskip

Boussa\"id, Caponigro and Chambrion \cite{BCC2}
recently extended this obstruction to the case of $L^1$ (and even Radon measures)
controls by considering the Dyson expansion of the solution. 
We show here that this technique can be adapted to the case of some nonlinear wave equations. 
This shows in particular that the  nonlinear term does not help to control the 
equation in its natural energy space. \medskip

We consider the following abstract control system
\begin{equation} \label{EQ_NLS_abstract}
\left\{
\begin{aligned}
&\psi'(t)=A\psi(t) + u(t) B \psi(t) + K(\psi(t)),\\
&\psi(0)=  \psi_0 \in \X,
\end{aligned}
\right.
\end{equation}
with real valued controls $u:\R \to \R$ and with the following 
assumptions.

\begin{assumption}\label{ASS_BMS_L1} The element  $(\X,A,B, K)$ satisfies
\begin{enumerate}[$(i)$]
\item $\X$ is Banach space endowed with norm $\| \cdot \|_{\X}$.  \label{ASS_X_Banach}
\item $A:D(A)\to \X$ is a linear operator with domain $D(A)\subset \X$ that 
generates a $C^0$ semi-group of bounded linear operators. We denote by $\omega \geq 0$ and $M >0$ two numbers such that $\|e^{tA}\|_{L(\X,\X)} \leq M e^{\omega t}$ for every $t\geq 0$. \label{ASS_A_generateur}
\item $B:\X\to \X$ 
is a linear bounded operator. \label{ASS_Bi_bounded}
\item $K:\X \to \X$ is  $k$-Lipschitz-continuous (not necessarily linear), with $k>0$.  \label{ASS_C_Lipschitz}
\end{enumerate}
\end{assumption}

In the sequel, the equation~\eqref{EQ_NLS_abstract} is interpreted in its mild form, namely, we say that a function $\psi:[0,T]\to \X$ is a solution 
of~\eqref{EQ_NLS_abstract} if, for every $t$ in $[0,T]$,
\begin{equation}\label{EQ_Duhamel}
\psi(t)= e^{tA}\psi_0  + \int_0^t \!\!\! u(s) e^{(t-s)A} B \psi(s)\mathrm{d}s + \int_0^t \!\!\!e^{(t-s)A} K(\psi(s)) \mathrm{d}s.
\end{equation}
 Equation~\eqref{EQ_Duhamel} is often called Duhamel formula.  

\subsection{Notations} Throughout the paper, for the sake of readability, we omit the range in the notation of spaces of real-valued functions. For instance, if $X$ is a space, $H^k(X)$ denotes the set of~$H^k$ regular real functions on 
$X$.

In a metric space $X$ endowed with distance $d_X$, we define the ball centered in $x \in X$ with radius $r>0$ by $B_X(x,r)=\{y \in X |d_X(x,y)<r\}$. If $X$ is a vector space endowed with norm $\|\cdot \|_X$, the distance associated with the norm is denoted $d_X$: $d_X(x,y)=\|x-y\|_X$, for every $x,y$ in $X$.

\subsection{Main result}
Under Assumption~\ref{ASS_BMS_L1}, one can show that equation~\eqref{EQ_NLS_abstract} admits a global flow $\Phi^u$ (see Propositions~\ref{PRO_22} and~\ref{PRO_Yinfty_solution_mild}). Our  main result concerning the control of~\eqref{EQ_NLS_abstract} gives a description of the attainable set and reads as follows

\begin{thm}\label{THM_BMS_abstract}
Let $(\X,A,B,K)$ satisfy Assumption~\ref{ASS_BMS_L1}. Then, for every $\psi_0$ in 
$\X$,  the attainable set 
from $\psi_0$ of~\eqref{EQ_NLS_abstract} with controls   $u$ in 
$L^1([0,+\infty))$, $ \bigcup_{t\geq 0} \bigcup_{u \in L^1([0,t])} \Big\{\Phi^u(t) \psi_0\Big\}$,
is contained in a countable union of compact subsets of $\X$.
\end{thm}

This result gives a clear obstruction to the controllability 
of~\eqref{ASS_BMS_L1} in a general setting, since it shows that the attainable 
set is meager in the sense of Baire. However, as noted by Beauchard and Laurent 
in~\cite[Section 1.4.1]{MR2732927}, this result does not exclude exact 
controllability in a smaller space, endowed with a stronger norm (for which the operator $B$ is not any longer continuous).  In this sense, this  obstruction to controllability may be seen as an 
unfortunate choice of the ambiant
space.  \medskip

The proof of Theorem~\ref{THM_BMS_abstract} relies on  the description of the solutions of~\eqref{EQ_NLS_abstract} by series, called Dyson expansion (see Section~\ref{Sect2}). This strategy has been successfully carried out for the case $K=0$ (linear dynamics) in \cite[Section 5.1]{BCC2}, and we show here that it can also be applied to nonlinear problems. For more details on Dyson expansions, we refer to \cite[Theorem X.69 and equation (X.129)]{RS2}. \medskip

In the assumptions of the Theorem~\ref{THM_BMS_abstract}, the fact that $K$ is Lipschitz is needed in order to ensure the existence of a global flow of~\eqref{EQ_NLS_abstract}, but in the core of the proof of our result we only need that~$K$ is continuous (see Proposition~\ref{PRO_Yj_compact}). \medskip

 We provide two explicit applications of Theorem~\ref{THM_BMS_abstract} to nonlinear wave equations.  We first give the example of the Sine-Gordon equation,  which exactly matches  Assumption~\ref{ASS_BMS_L1} and 
 to which Theorem~\ref{THM_BMS_abstract} directly applies. Then,   by  means of the 3-dimensional cubic Klein-Gordon equation, we show that the hypothesis ``$K$ is Lipschitz'' can be relaxed. Actually,  for the nonlinear wave equation (see Section~\ref{KG}), the gain of derivative in the Duhamel formula allows one to bound the nonlinearity using Sobolev estimates, and the global existence of a flow can be obtained by energy estimates. \medskip
 
We  are also able to obtain negative controllability  results for the nonlinear
Schr\"odinger equation, and   this  will be treated in our 
forthcoming paper \cite{ChTh2}.
\begin{rem}
By rather simple modifications, the result of Theorem~\ref{THM_BMS_abstract} can be extended to the case of the equation 
\begin{equation*} 
\psi'(t)=A\psi(t) + \sum_{j=1}^nu_j(t) B_j \psi(t) + \alpha(t)K(\psi(t)),
\end{equation*}
with the same assumptions on the controls $u_j\in L^1([0,+\infty))$ and with  $\alpha \in L^1([0,+\infty))$ being given. Such models are relevant in some physical contexts (e.g. the Schr\"odinger equation with electric and magnetic fields combined with coupling to the environment in the spirit of \cite{Hansom}), but we omit the details  to simplify the presentation. 
\end{rem}

%%%%%%%%%%%%%%%%%%%%%%%%%%%%%%%%%%%%%%%%%%%%%%%%%%%%%%%%%%%%%
%%%%%%%%%%%%%%%%%%%%%%%%%%%%%%%%%%%%%%%%%%%%%%%%%%%%%%%%%%%%%
\section{Ball-Marsden-Slemrod obstructions for nonlinear equations}\label{Sect2}

%%%%%%%%%%%%%%%%%%%%%%%%%%%%%%%%%%%%%%%%%%%%%%%%%%%%%%%%%%%%%
\subsection{Dyson expansion of the solutions}

Let $T>0$ and $u$ be given in $L^1([0,T])$. Define by induction on $p\geq 0$, 

\begin{equation}\label{DEF_Yp}
\left \{
  \begin{aligned}
& Y_{0,t}^u\psi_0=0\\
& Y_{p+1,t}^u \psi_0 =e^{tA}\psi_0 + \int_0^t e^{(t-s)A} \left \lbrack  u(s) B Y_{p,s}^u \psi_0
 + K (Y^u_{p,s}\psi_0) \right \rbrack \mathrm{d}s
      \end{aligned}
\right.
\end{equation}
 and $Z_{p,t}^u \psi_0= Y_{p+1,t}^u \psi_0-Y_{p,t}^u \psi_0$.\medskip

We aim to show that the series $(\sum_p Z_{p,t}^u \psi_0)$ converges. To this end, we need some quantitative bounds, which are stated in the next result.

\begin{proposition}\label{PRO_majoration_Z}
For every $j$ in $\mathbf{N}$, every $t>0$ and every $u$ in $L^1([0,+\infty))$,
\begin{equation}\label{EQ_maj_norme_Z}
\|Z_{j,t}^u \psi \|_{\X} \leq  \frac{ e^{\omega t}M^{j+1} \left (kt+  \|B\|_{L(\X,\X)} \int_0^t |u(s)| \mathrm{d}s  \right )^j}{j!} \|\psi\|_{\X} . 
\end{equation}
\end{proposition}

\begin{proof} We proceed by induction on $j \geq 0$. 
The inequality~\eqref{EQ_maj_norme_Z} for $j=0$ follows from Assumption~$\ref{ASS_BMS_L1} (\ref{ASS_A_generateur})$. Assume now that we have proved~\eqref{EQ_maj_norme_Z}  for a given $j$. Then, since
\begin{eqnarray*}
\|Z_{j+1,t}^u \psi \|_{\X} & \leq &  \int_0^t M e^{\omega (t-s)} \big(k+ |u(s)| \|B\|_{L(\X,\X)} \big) \|Z_{j,s}^u \psi \|_{\X} \mathrm{d}s\\
& \leq & \frac{M^{j+2}}{j!}  e^{\omega t} \left \lbrack  \int_0^t   \big(k+ |u(s)| \|B\|_{L(\X,\X)}\big)  \Big(k {s}+   \|B\|_{L(\X,\X)}\int_0^s |u(\tau)| \mathrm{d}\tau  \Big)^{j} \mathrm{d}s \right \rbrack \|\psi\|_{\X}\\
&\leq & \frac{M^{j+2}} {(j+1)!}  e^{\omega t} \Big (k {t}+   \|B\|_{L(\X,\X)}\dis \int_0^t |u(s)| \mathrm{d}s  \Big)^{j+1} \|\psi\|_{\X} ,
\end{eqnarray*}
which concludes the proof.
\end{proof}

From Proposition~\ref{PRO_majoration_Z}, for every $t$ in $[0,T]$ and every $\psi$ in $\X$, the sum $\sum_j Z_{j,t}^u \psi$ converges in $\X$. We denote this sum by $Y_{\infty,t}^u \psi$:
$$Y_{\infty,t}^u \psi= \sum_{j=0}^{+\infty} Z_{j,t}^u \psi $$

\begin{proposition}\label{PRO_22}
For every $\psi$ in $\X$, every $T>0$ and every $u$ in $L^1([0,+\infty),\mathbf{R})$, the function $(t,\psi)\mapsto Y_{\infty,t}^u \psi$ is continuous
from $\mathbf{R}\times \X$ to $\X$. 
\end{proposition}
\begin{proof}
This follows from the continuity of the functions $(t,\psi)\mapsto Z_{j,t}^u \psi$ for every $j\geq 0$ and from the convergence of $\sum_j Z_{j,t}^u \psi$ (locally uniform in $t$ and $\psi$)  from Proposition~\ref{PRO_majoration_Z}.
\end{proof}

\begin{proposition}\label{PRO_Yinfty_solution_mild}
For every $T \in [0,+\infty)$, every $u$ in $L^1([0,T],\mathbf{R})$ and every $\psi_0$ in $\X$, $t\mapsto Y_{\infty,t}^u \psi_0$ is the unique mild solution on $[0,T]$  of~\eqref{EQ_NLS_abstract} taking value $\psi_0$ at $0$.
\end{proposition}
\begin{proof}
 The mapping
\[
\begin{array}{llcl}
F:&C^0([0,T],\X) & \longrightarrow & C^0([0,T],\X)\\
& \big(t\mapsto \psi(t)\big) & \longmapsto & \big( t\mapsto e^{tA}\psi_0+\int_0^t e^{(t-s)A} \left \lbrack  u(s) B \psi 
 + K (\psi) \right \rbrack \mathrm{d}s\big)
\end{array}
\]
is continuous for the norm $L^{\infty}([0,T],\X)$. By~\eqref{DEF_Yp}, $t\mapsto Y_{\infty,t}^u \psi_0$ is a fixed point of $F$, hence a mild solution on $[0,T]$  of~\eqref{EQ_NLS_abstract} taking value $\psi_0$ at $0$.

Assume that $t\mapsto \psi_1(t)$ and  $t\mapsto \psi_2(t)$ are two mild solutions 
on $[0,T]$  of~\eqref{EQ_NLS_abstract} taking value~$\psi_0$ at~$0$. 
Define $T^\ast=\dis\sup_{t \in [0,T]} \big\{t \; |\; \psi_1(s)=\psi_2(s), \; {\mbox{for almost every }} s\leq t\big\}$. We will prove by 
contradiction that $T^\ast=T$, that is, $\psi_1=\psi_2$ {almost everywhere}. Assume that $T^\ast<T$.
We chose $t_1 \in (T^\ast,T]$ such that
\[
 M e^{(t_1-T^\ast) \omega} \left (k(t_1-T^\ast) + \|u\|_{L^1([T^\ast,t_1],\mathbf{R})} \|B\|_{L(\X,\X)} \right ) := C_0<1.
 \]
 Then, for all $T^\ast \leq t_2 \le t_1$
 \begin{eqnarray*}
 \lefteqn{\|\psi_2(t_2)-\psi_1(t_2)\|_{\X}}\\&=&\left \| \int_{T^\ast}^{t_2}  u(s) e^{(t_2-s)A} B (\psi_2(s)-\psi_1(s))\mathrm{d}s + \int_{T^\ast}^{t_2} e^{(t_2-s)A} \Big(K(\psi_2(s))-K(\psi_1(s))\Big) \mathrm{d}s \right \|_{\X}\\
&\leq &  \int_{T^\ast}^{t_2}   |u(s)| M e^{(t_2-s)\omega}  \|B\|_{L(\X,\X)} \|\psi_2(s)-\psi_1(s) \|_{\X} \mathrm{d}s + \int_{T^\ast}^{t_2} M e^{(t_2-s)\omega} k \|\psi_2(s)-\psi_1(s) \|_{\X}\mathrm{d}s \\
 &\leq & \|\psi_2-\psi_1 \|_{L^\infty([T^\ast,t_1),\X)} \int_{T^\ast}^{t_1}  \left ( k+ |u(s)| \|B\|_{L(\X,\X)}\right ) M e^{(t_1-s)\omega}  \mathrm{d}s \\
 &\leq &C_0   \|\psi_2-\psi_1 \|_{L^\infty([T^\ast,t_1),\X)},
 \end{eqnarray*}
 therefore we deduce that 
 $$  \|\psi_2-\psi_1 \|_{L^\infty([T^\ast,t_1),\X)} \leq C_0   \|\psi_2-\psi_1 \|_{L^\infty([T^\ast,t_1),\X)},$$
 which gives the desired contradiction.  To conclude the proof, it remains to show that any mild solution is continuous (since two continuous functions coincide as soon as they are equal almost everywhere). And indeed, any mild solution solution of~\eqref{EQ_NLS_abstract} is  equal almost everywhere to $Y^u_\infty$, which is  continuous  (Proposition \ref{PRO_22}), hence any mild solution of~\eqref{EQ_NLS_abstract} is essentially bounded and then is continuous by its definition~\eqref{EQ_Duhamel}.
\end{proof}

\begin{definition}
Let $T>0$, $u$ in $L^1([0,+\infty),\mathbf{R})$ and $\psi_0$ in $\X$.
In the following, we denote by $t\mapsto \Phi^u(t) \psi_0$  the mild solution of system~\eqref{EQ_NLS_abstract} associated with the initial condition $\psi_0$ and the control $u$ in $L^1([0,T))$. 
\end{definition}

We sum up the above results in the following
\begin{proposition}[Dyson  expansion of the solutions of~\eqref{EQ_Duhamel}]\label{PRO_Dyson_expansion}
Let $t>0$, $u$ in $L^1([0,+\infty),\mathbf{R})$ and $\psi_0$ in $\X$. Then 
\begin{equation}\label{Dyson}
\Phi^u(t)\psi_0 = \sum_{j=0}^\infty Z_{j,t}^u(\psi_0).
\end{equation}
\end{proposition}
 
%%%%%%%%%%%%%%%%%%%%%%%%%%%%%%%%%%%%%%%%%%%%%%%%%%%%%%%%%%%%
\subsection{A compactness result}\label{SEC_compctness_result_Lipschitz}

Recall that $Y_{j,t}^u \psi_0$ is defined in~\eqref{DEF_Yp} and that $Z_{j,t}^u \psi_0= Y_{j+1,t}^u \psi_0-Y_{j,t}^u \psi_0$.

\begin{proposition}\label{PRO_Yj_compact}
For every $j$ in $\mathbf{N}$, $T \geq  0$  and $L \geq 0$, and $\psi_0$ in $
\X$, the sets
\[ \mathcal{Z}_j^{T,L} = \left \{Z_{j,t}^u \psi_0 \; | \; 0 \leq t \leq T, \|u\|_{L^1(0,T)} 
\leq L \right\}  \mbox{ and }\; \mathcal{Y}_j^{T,L} = \left \{Y_{j,t}^u \psi_0 \; | \; 0 
\leq t \leq T, \|u\|_{L^1(0,T)} \leq L \right\} 
 \]
are relatively compact in $\X$.
\end{proposition}

\begin{proof}
We adapt the proof of \cite{BCC2} (valid for $K=0$) to the general case of a continuous function $K$.
  
Since a finite  sum of relatively compact sets is still relatively compact, it is enough 
to prove the result for $\mathcal{Y}_j^{T,L}$.
We prove this by induction on $j \geq 0$.

For $j=0$, the result is clear.  

Assume that $\mathcal{Y}_j^{T,L}$ is relatively 
compact in $\X$ for some $j \geq 0$. We aim to prove that $\mathcal{Y}_{j+1}^{T,L}$ is
 relatively compact in $\X$ as well. For this, we chose $\varepsilon>0$ and we try to exhibit an $\varepsilon$-net of $\mathcal{Y}_{j+1}^{T,L}$.

Since the mappings 
\[
\begin{array}{llcl}
G_1:&[0,T]\times \X &\longrightarrow & \X\\
&(s,\psi) & \longmapsto & e^{(T-s)A}   B \psi
\end{array}
 \mbox{ \quad and \quad }
\begin{array}{llcl}
G_2:&[0,T]\times \X &\longrightarrow & \X\\
&(s,\psi) & \longmapsto & e^{(T-s)A} 
  K (\psi) 
\end{array}
\]
are continuous,
the sets $G_1([0,T]\times \mathcal{Y}_{j}^{T,L})$ and $G_2([0,T]\times \mathcal{Y}_{j}^{T,L})$ are relatively compact.  
Hence, there exists a finite family $(x_i)_{1\leq i \leq N}$ such that, for 
$\ell=1,2$, 
\[G_\ell([0,T]\times \mathcal{Y}_{j}^{T,L}) \subset \bigcup_{i=1}^N B_{\X}\left (x_i, 
\frac{\varepsilon}{ 4(L+T)} \right).  \]
Let $(\phi_i)_{1\leq i \leq N}$ be a partition of unity associated with the 
 covering of $G_\ell([0,T]\times \mathcal{Y}_{j}^{T,L})$, $\ell=1,2$. That is, 
 the functions $\phi_i$ satisfy $0 \leq \phi_i \leq 1$ and,   for
 every $x$ in $G_\ell([0,T]\times \mathcal{Y}_{j}^{T,L})$,
 $\displaystyle{\sum_{i=1}^N 
\phi_i(x)=1}$ and $\dis \Big\| x-\sum_{i=1}^N \phi_i(x) x_i \Big\|_{\X} <  \frac{\varepsilon}{2(L+T)}$. 
%  (for the 
% construction of such a partition of unity, we refer for instance to 
% \cite[Section II]{BCC2}). 

Then, for every $u$ in $L^1([0,T],\mathbf{R})$ such that $\|u\|_{L^1(0,T)}\leq L$, 
\[ \Big \| \int_0^t u(s) e^{(t-s)A} ( B Y^u_{j,s}\psi_0 ) 
\mathrm{d}s -\sum_{i=1}^N \int_0^t  u(s) \phi_i\big(e^{(t-s)A} (  B Y^u_{j,s}\psi_0 \big) x_i \mathrm{d}s \Big \|_{\X}  \leq \frac{L\varepsilon}{2(L+T)},  \]
 and \[ \Big \| \int_0^t  e^{(t-s)A} K( Y^u_{j,s}\psi_0 ) 
\mathrm{d}s -\sum_{i=1}^N \int_0^t   \phi_i(e^{(t-s)A} K( Y^u_{j,s}\psi_0 )) x_i \mathrm{d}s \Big \|_{\X} \leq  \frac{ T\varepsilon }{2(L+T)}. \]
Now  using that  the compact sets $\sum_{i=1}^N [0,L] x_i$ and $\sum_{i=1}^N [0,T] x_i$ admit a $\varepsilon/4$-net $(y_i)_{1\leq i \leq N_2}$, and  the previous estimates, we get 
$\displaystyle{  \mathcal{Y}_{j+1}^{T,L} \subset \bigcup_{i=1}^{N_2} B_{\X}(y_i,\varepsilon)}$,  which concludes the proof. 
\end{proof}

\begin{rem}
In the proof of Proposition~\ref{PRO_Yj_compact}, we  only used the continuity of $K$. In this paper, we assume that $K$ is Lispschitz continuous in order to ensure the global existence of a flow of~\eqref{EQ_Duhamel} and the Dyson expansion~\eqref{Dyson}. In Section~\ref{KG}, we will show that our approach applies to more general nonlinearities, which are only locally (not globally) Lipschitz continuous.
\end{rem}

%%%%%%%%%%%%%%%%%%%%%%%%%%%%%%%%%%%%%%%%%%%%%%%%%%%%%%%%%%%%%
\subsection{Proof of the nonlinear Ball-Marsden-Slemrod obstructions}\label{SEC_proof_BMS_Lipschitz}

We are now able to complete the proof of Theorem~\ref{THM_BMS_abstract}. \medskip

For  $T>0$ and $L>0$ define
\[
\mathcal{V}^{T,L}= \big\{ \Phi^u(t) \psi_0 \;| \; u \in L^1([0,T]), \|u\|_{L^1([0,T])} \leq L, 0 \leq t \leq T \big\},
\]
and notice that 
\[
\bigcup_{t\geq 0} \bigcup_{u \in L^1} \{\Phi^u(t) \psi_0\}
=
\bigcup_{T \in  \mathbf{N}} \bigcup_{L \in \mathbf{N}}
\mathcal{V}^{T,L}.
\]
Thus it is enough to prove that, for every $T>0$ and every $L>0$, the set $\mathcal{V}^{T,L}$ is relatively compact.

Let $\delta>0$ be given. We aim to find a $\delta$-net of $\mathcal{V}^{T,L} $.

From Propositions~\ref{PRO_majoration_Z} and~\ref{PRO_Dyson_expansion}, since $\dis \big\| \sum_{j=N}^\infty Z^u_{j,T} (\psi_0)\big\|_{\X}$ tends to zero as $N$ tends to infinity uniformly with respect to $u$ in $B_{L^1([0,T],\mathbf{R})}(0,L)$, there exists $N_1$ large enough such that, for every  $u$ in $B_{L^1([0,T],\mathbf{R})}(0,L)$,
\[
\Big \| \sum_{j=N_1}^\infty Z^u_{j,T} (\psi_0)\Big \|_{\X} < \frac{\delta}{2}.
\] 
The set $\mathcal{Y}_{N_1}^{T,L}$ is relatively compact (Proposition~\ref{PRO_Yj_compact}); hence it admits a $\delta/2$-net. 

Thus  \[
\mathcal{V}^{T,L} \subset \Big \{x \in \X\, | \,d_{\X}\Big (x,\mathcal{Y}_{N_1}^{T,L} \Big ) \leq \frac{\delta}{2} \,\Big \} \]
admits a $\delta$-net, which finishes the proof.  \hfill $\square$
 
%%%%%%%%%%%%%%%%%%%%%%%%%%%%%%%%%%%%%%%%%%%%%%%%%%%%%%%%%%%%%
%%%%%%%%%%%%%%%%%%%%%%%%%%%%%%%%%%%%%%%%%%%%%%%%%%%%%%%%%%%%%
\section{Applications}

%%%%%%%%%%%%%%%%%%%%%%%%%%%%%%%%%%%%%%%%%%%%%%%%%%%%%%%%%%%%
\subsection{The Sine-Gordon equation}
 
 We consider the Sine-Gordon equation which reads
    \begin{equation}\label{Sine-G}
  \left\{
      \begin{aligned}
& \partial_t^2\psi-\partial^2_x\psi=u(t)B(x) \psi-\sin\psi, \quad (t,x) \in \R\times \R, \\
        & \psi(0,.)=\psi_0 \in H^1(\R),   \\
           &  \partial_t\psi(0,.)=\psi_1 \in L^2(\R), 
      \end{aligned}
    \right.
\end{equation}
where $B$ is a given function, and $u\in L_{loc}^1(\R)$ is the control. In the case   $B \equiv 0$, this equation appears  in  relativistic field theory or in the study of mechanical transmission lines. We rewrite this equation as a first order (in time) system, so that it fits the framework of our study. Equation~\eqref{Sine-G}  is equivalent to

    \begin{equation*} 
  \left\{
      \begin{aligned}
&\partial_t\begin{pmatrix}
\psi \\ \phi
\end{pmatrix}
=
\begin{pmatrix}
0&1\\\partial^2_x&0
\end{pmatrix}
\begin{pmatrix}
\psi \\ \phi
\end{pmatrix}
+u(t)B(x)   \begin{pmatrix}
0&0\\1&0
\end{pmatrix}\begin{pmatrix}
\psi \\ \phi
\end{pmatrix}+\begin{pmatrix}
0\\ -\sin\psi
\end{pmatrix},  \quad (t,x) \in \R\times \R, \\
        & (\psi(0,.), \phi(0,.))  =(\psi_0,\psi_1)  \in H^1(\R) \times L^2(\R).
      \end{aligned}
    \right.
\end{equation*}
Then Theorem~\ref{THM_BMS_abstract} directly applies with $\X=H^1(\R) \times L^2(\R)$, $A= \begin{pmatrix}
0&1\\\partial^2_x&0
\end{pmatrix}$, $D(A)=H^2(\R) \times H^1(\R)$, $B\in L^{\infty}(\R)$ and $K(\psi, \phi)=(0; -\sin(\psi))$.

   %%%%%%%%%%%%%%%%%%%%%%%%%%%%%%%%%%%%%%%%%%%%%%%%%%%%%%%%%%%%%
  \subsection{The wave equation in dimension 3}\label{KG}
 The result of Theorem~\ref{THM_BMS_abstract} also applies to nonlinear equations, with local Lipschitz nonlinear terms. We develop here the examples of the wave and Klein-Gordon equations.
 Denote by $\mathcal{M}$ a  compact manifold of dimension 3 without boundary, or $\mathcal{M}=\R^3$.   We consider the defocusing cubic  wave equation 
   \begin{equation}\label{kg}
  \left\{
      \begin{aligned}
      & \partial_t^2\psi-\Delta\psi+m \psi=u(t)B(x) \partial_t\psi-\psi^3, \quad (t,x) \in \R\times \mathcal{M}, \\
        & \psi(0,.)=\psi_0 \in H^1(\mathcal{M}),   \\
           &  \partial_t\psi(0,.)=\psi_1 \in L^2(\mathcal{M}), 
      \end{aligned}
    \right.
\end{equation}
 with $m \geq 0$ and $B \in L^{\infty}(\mathcal{M})$. Positive exact controllability results for such non-linear dynamics   in the case $\mathcal{M}=(0,1)$
  were obtained by Beauchard \cite[Theorem 1]{Beauchard11}.

 Let the control function $u$ be in $ L_{loc}^1(\R)$;  the mild solution reads
  \begin{equation*}
\psi(t)=S_0(t)\psi_0+S_1(t) \psi_1+\int_0^tS_1(t-s) \big(u(s)B(x) \partial_s\psi(s)-\psi^3(s)\big)ds
 \end{equation*}
 where 
   \begin{equation}\label{defS}
 S_0(t)=\cos(t \sqrt{-\Delta+m} ) \;\;\text{ and }\;\; S_1(t)=\frac{\sin(t \sqrt{-\Delta+m} )}{ \sqrt{-\Delta+m}}. 
 \end{equation}

  \subsubsection{The obstruction result for controllability of the wave equation} 
  
  We state the main result of this section, which is the analogue of Theorem~\ref{THM_BMS_abstract} for equation~\eqref{kg}.
  
   \begin{thm}\label{THM_BMS_KG}
For all $(\psi_0,\psi_1)\in H^1(\mathcal{M}) \times L^2(\mathcal{M})$ and $u \in L^1(\R)$,  there exists a unique solution to~\eqref{kg}
\begin{equation*}
\psi\in \mathcal{C}^{0}\big(\R;H^{1}(\mathcal{M})\big)\cap \mathcal{C}^{1}\big(\R;L^{2}(\mathcal{M})\big).
\end{equation*}
 This enables us to define a global flow 
 \[ \begin{array}{llcl}
 \Phi=(\Phi_1,\Phi_2):& H^1(\mathcal{M}) \times L^2(\mathcal{M}) \times L^1(\R) & \longrightarrow &
 \mathcal{C}^{0}\big(\R;H^{1}(\mathcal{M})\big) \times \mathcal{C}^{0} (\R;L^{2}(\mathcal{M})\big)\\
 &(\psi_0,\psi_1,u) & \longmapsto &  (\psi, \psi_t)
 \end{array}.\]  
 
Moreover, for every $(\psi_0,\psi_1) \in H^1(\mathcal{M}) \times L^2(\mathcal{M})$, the  attainable set 
$$
\bigcup_{t \in \R} \bigcup_{u \in L^1} \big\{\Phi^u(t) (\psi_0,\psi_1)\big\}
$$
is contained in a countable union of compact subsets of $H^1(\mathcal{M}) \times L^2(\mathcal{M})$.
\end{thm}

While we decided to illustrate our method for the equation~\eqref{kg}, our approach can be applied to other wave-type equations, such as  
\[ \partial_t^2\psi-\Delta\psi+m \psi=u(t)B(x)  \psi-\psi^3, \]
with a given potential  $B \in L^{3}(\M)$. We omit the details.

  \subsubsection{Local and global existence results} 

Since  equation~\eqref{kg} is reversible, in the sequel, we onöy consider non-negative times.   Let $T>0$,  $u \in L^1([0,T])$ and $t_0\geq 0$ be given.  We define by induction on $p\geq 0$, 
\begin{equation*} 
\left \{
  \begin{aligned}
& \widetilde{Y}_{0,t,t_0}^u=0\\
& \widetilde{Y}_{p+1,t,t_0}^u(\psi_0,\psi_1)  =S_0(t)\psi(t_0)+S_1(t) \partial_t\psi(t_0) + \int_{0}^{t} \!\!\!S_1(t-s) \left \lbrack  u(s+t_0) B(x) \partial_s\widetilde{Y}_{p,s,t_0}^u
 - (\widetilde{Y}^u_{p,s,t_0})^3 \right \rbrack \mathrm{d}s
      \end{aligned}
\right.
\end{equation*}
with $\widetilde{Y}_{p,s,t_0}^u=\widetilde{Y}_{p,s,t_0}^u(\psi_0,\psi_1)$, and where $S_0$ and $S_1$ are defined in~\eqref{defS}. \medskip

We now  state a global existence result, which is an application of the Picard fixed point theorem.
 
 \begin{proposition}\label{GWP-KG}
 \begin{enumerate}[$(i)$]
\item For all $(\psi_0,\psi_1)\in H^1(\mathcal{M}) \times L^2(\mathcal{M})$ there exists a unique solution to~\eqref{kg}
\begin{equation*}
\psi\in \mathcal{C}^{0}\big(\R;H^{1}(\mathcal{M})\big)\cap \mathcal{C}^{1}\big(\R;L^{2}(\mathcal{M})\big).
\end{equation*}
\item Moreover,  for all $T>0$, for all $L>0$ and $u$ such that $\int_0^T|u(s)|ds\leq L$,
$$\dis \sup_{0 \leq t \leq T}
 \left \|(\psi,\partial_{t}\psi)(t)\right \|_{H^1(\mathcal{M}) 
 \times L^2(\mathcal{M})}\leq  C\left(\|\psi_0\|_{H^1}, \|\psi_1\|_{L^2}, L,T\right),$$
 where $C$ is a  continuous  function.
\item  Furthermore, for all $T>0$, and $L>0$, there exist $k \geq 1$, $0<c_0<1$ and a 
continuous   function 
$\tau=\tau(\|\psi_0\|_{H^1}, \|\psi_1\|_{L^2}, L,T)>0$ 
such that, for all $0\leq t_0 \leq T$,  $p\geq 0$ and $u$ 
with $\int_0^T|u(s)|ds\leq L$,
    \begin{equation}\label{picard}
     \sup_{ t  \in [0, \tau]}\|\big(\psi(t+t_0)-\widetilde{Y}_{kp,t,t_0}^u,\partial_{t}\psi(t+t_0)-\partial_{t}\widetilde{Y}_{kp,t,t_0}^u\big)\|_{H^1(\mathcal{M}) \times L^2(\mathcal{M})}\leq  Cc_0^p.
     \end{equation} \label{3.2.iii}
 \end{enumerate}
 \end{proposition} 
 
 In the previous result, it is crucial that we obtain a  time $\tau=\tau(
\|\psi_0\|_{H^1}, \|\psi_1\|_{L^2}, L,T)$ which only depends on the 
 norms of $\psi_0$, $\psi_1$ and $u$ (and not $\psi_0$, $\psi_1$ or $u$ themselves). This fact will be used in the compactness argument (see Section~\ref{SEC_compactness_argument_KG}).

\begin{proof}    {\it A first local existence result:}  Let $t_0\geq 0$. We prove a local in time  existence result for the problem 
 \begin{equation}\label{kgt0}
  \left\{
      \begin{aligned}
      & \partial_t^2\widetilde{\psi}-\Delta\widetilde{\psi}+m \widetilde{\psi}=u(t)B(x)\partial_t \widetilde{\psi}-\widetilde{\psi}^3, \quad (t,x) \in \R\times \mathcal{M}, \\
        & \widetilde{\psi}(t_0,.)={\widetilde{\psi}}_0 \in H^1(\mathcal{M}),   \\
           &  \partial_t\widetilde{\psi}(t_0,.)=\widetilde{\psi}_1 \in L^2(\mathcal{M}).
      \end{aligned}
    \right.
\end{equation}
We consider the map 
  \begin{equation*} 
F(\psi)(t)=S_0(t)\widetilde{\psi}_0+S_1(t)\widetilde{\psi}_1 + \int_{0}^{t} S_1(t-s) \left \lbrack  u(s+t_0) B(x)\partial_s \psi(s)
 - (\psi(s))^3 \right \rbrack \mathrm{d}s,
 \end{equation*} 
and we will show that, for $t>0$ small enough, it is a contraction in some 
Banach space. Then, by the Picard theorem, there will exist a unique fixed point 
$\psi$, and $\widetilde{\psi}(t)=\psi(t-t_0)$ will be the unique solution 
to~\eqref{kgt0}. \medskip

We define the norm 
$\|\psi\|_T= \| \psi \|_{L^{\infty}_TH^1}+\| \partial_t\psi \|_{L^{\infty}_TL^2}$ 
and the space 
  \begin{equation*}
X_{T,R}= \big\{ \|\psi\|_T\leq R\big\},
 \end{equation*} 
 with $R>0$ and $T>0$ to be determined.

By  the  Sobolev embedding $H^1(\mathcal{M}) \subset L^6(\mathcal{M})$ (see 
Proposition~\eqref{PRO_Sobolev_embedding} with $p=2$ and $n=3$),  there exists $c=c(m, T)>0$ such that 
   \begin{eqnarray}\label{ineq37}
 \| F(\psi)\|_T &\leq& 2 (\|\widetilde\psi_0\|_{H^1}+ \|\widetilde\psi_1\|_{L^2})+
c \int_0^T\big( \|  u(s+t_0)  B\partial_s \psi \|_{L^2}+  \|  \psi(s) \|^3_{L^{6}
 }\big)ds 
  \nonumber \\
 &\leq& 2 (\|\widetilde\psi_0\|_{H^1}+ \|\widetilde\psi_1\|_{L^2})+c\big(\int_0^T  |
 u(s+t_0)| ds \big)\big\| B \big\|_{L^{\infty}} \big\| \partial_t\psi \big\|_{L_T^{\infty}L^2}+cT 
 \|  \psi \|^3_{L_T^{\infty}H^1} . 
 \end{eqnarray} 
Let us set   $R =4( \|\widetilde\psi_0\|_{H^1}+ \|\widetilde\psi_1\|_{L^2})$. 
Then we choose  $T_1=c_1 R^{-2}$ with $c_1>0$ small enough such that 
$cT_1 R^2 \leq 1/4$ and we choose $T_2>0$ such that  
$c\int_0^{T_2}  |u(s+t_0)| ds \leq \big\| B \big\|^{-1}_{L^{\infty}}/4$. Therefore, for 
$T=\min{(T_1,T_2)}$, $F$ maps $X_{T,R}$ into itself. With similar estimates we 
can show that $F$ is a contraction in $X_{T,R}$, namely 
    \begin{equation*}
 \| F(\psi_1)- F(\psi_2)\|_T  
 \leq  \big[cT R^2 + c\big(\int_0^T  |u(s+t_0)| ds \big)\big\| B \big\|_{L^{\infty}} 
 \big]  \| \psi_1- \psi_2\|_T.   
  \end{equation*} 
  As a consequence, there exists a unique, local in time solution to~\eqref{kgt0}, 
  with the time of existence, $\tau$, depending on  the norms of 
  $\widetilde\psi_0$, $\widetilde\psi_1$ and $u$.   \medskip

  {\it Energy bound:}   We define
\begin{equation*}
E(\psi)(t)=\frac12 \int_{\mathcal{M}} \big((\partial_t \psi)^2+|\nabla \psi|^2 
+m \psi^2\big)+\frac14 \int_{\mathcal{M}}  \psi^4.
\end{equation*}
%Since the time of existence only depends on $E(\psi)(0)$ 
By differentiation with respect to time, we get 
    \begin{eqnarray*} 
         \frac{d}{dt}  E(\psi)(t) &=   &  \int_{\mathcal{M}} \partial_t{ \psi}
         \big(\partial_t^2\psi-\Delta \psi +m\psi+ \psi^3   \big)dx \nonumber \\
       &=& u(t) \int_{\mathcal{M}} \partial_t{ \psi}. B  \partial_t{ \psi}   dx .
                  \end{eqnarray*}
Next, since $B \in L^{\infty}(\mathcal{M})$, we get  
    \begin{eqnarray*} 
              \frac{d}{dt}  E(\psi)(t)   &\leq    & 
               |u(t)|  \| B\|_{L^{\infty}} \| \partial_t \psi\|^2_{L^2} \\
        &\leq & C |u(t)| E(\psi)(t)
                    \end{eqnarray*}
  which implies 
   \begin{equation}\label{boundEnergy2}
   E(\psi)(t) \leq E(\psi)(0) \e^{C\int_0^t |u(s)|ds}.
      \end{equation}
 In the particular case $m=0$, the energy $E$ does not control the term $\int_{\mathcal{M}} \psi^2$, and we bound this latter term as follows. We set $M(\psi)(t)=\big(\int_{\mathcal{M}} \psi^2\big)^{1/2}$. Then
 $$  \frac{d}{dt}  M(\psi)(t)  \leq \| \partial_t \psi \|_{L^2} \leq  2 E^{1/2}(\psi)(t),$$
and by integration in time together with~\eqref{boundEnergy2} we obtain 
   \begin{equation}\label{boundEnergy}
   E(\psi)(t)+ \int_{\mathcal{M}} \psi^2 \leq  C_0(t, \int_0^t |u(s)|ds,  \|\psi_0\|_{H^1},  \|\psi_1\|_{L^2}).
      \end{equation}

   {\it Proof of $(i)$ and $(ii)$:}  Assume that one can solve~\eqref{kgt0} on 
   $[0,T^{\star})$, starting from $t_0=0$. By~\eqref{boundEnergy}, there is a   
   time $T_1^{\star}>0$ such that $c T_1^{\star}   (R^{\star})^2 \leq 1/4$ with 
   $R^{\star} =4c (\|\psi\|_{L^{\infty}_{T^{\star}}H^1}+\|\partial_t\psi\|_{L^{\infty}_{T^{\star}}L^2})$. Then we choose
   $T^{\star}_2>0$ such that  
    \[\dis {c\left ( \int_{T^{\star}-\frac{T_2^{\star}}{2}}^{T^{\star}+
    \frac{T_2^{\star}}{2}}  |u(s)| ds \right ) \big\| B  \big\|_{L^{\infty}}  \leq 1/4}.\]
      As a consequence, using the arguments of the previous (local) step, we are able 
      to solve the equation~\eqref{kgt0}, with an initial condition at 
      $t_0=T^{\star}-{\min(T_1^{\star},T_2^{\star})}/{2}$, on the  time interval 
      $[T^{\star}-{\min(T_1^{\star},T_2^{\star})}/{2},T^{\star}+
      {\min(T_1^{\star},T_2^{\star})}/{2}]$. 
      This shows that the maximal solution is global in time.  \medskip
   
      {\it Proof of $(iii)$:} To prove this last statement, we will find a time 
      of existence which does not depend on $t_0 \in [0,T]$ and which only 
      depends on $u$ through the quantity $\int_0^T|u(s)|ds$.  
      Assume that $\int_0^T|u(s)|ds\leq L$.
   
 For $k\geq 0$,   we denote by $F^k=F \circ F \circ \cdots \circ F$ the $k$th 
 iterate of $F$. From~\eqref{bornek} and~\eqref{contratk} (see Lemma~\ref{lemk} below) we obtain the bounds (with $L=L(T)$)
   \begin{equation*}\label{bornek1}
 \|F^k(\psi)\|_{T_1} \leq C_k(L, \|\psi\|_0)+\frac{\big(CL\big)^{k}}{k!}  
 \|\psi\|_{{T_1}}+T_1P_k(T_1,L, \|\psi\|_{T_1})
 \end{equation*}
 and 
  \begin{equation*}\label{contratk1}
 \|F^k(\psi)-F^k(\phi)\|_{T_1} \leq  \Big[\frac{\big(CL\big)^{k}}{k!} + 
 T_1Q_k(T_1,L, \|\psi\|_{{T_1}},\|\phi\|_{T_1} )        
 \Big] \|\psi-\phi\|_{{T_1}}.
 \end{equation*}
 Set $k \geq 0$ such that $\frac{(CL)^{k}}{k!}  \leq 1/2$.  Let 
 $R_1=\max\big(2C_k, C_0\big) $, where $C_k=C_k(L, \|\psi\|_0)$ is given in~\eqref{bornek} and  $C_0=C_0(T,L,\|\psi\|_0)$ is given in \eqref{boundEnergy}. Set 
   \begin{equation*}
X_{T_1,R_1}= \big\{ \|\phi\|_{T_1}\leq R_1\big\}.
 \end{equation*} 
  Then from the two previous estimates we infer that 
  $F^k : X_{T_1,R_1} \longrightarrow X_{T_1,R_1}$ is a contraction, provided that 
  $T_1=T_1(L, R_1)$ is small enough. As a consequence, there exists a unique 
  solution in $X_{T_1,R_1}$ to the equation $\phi=F^k(\phi)$. However, it is not clear whether 
  $F$ does  map $X_{T_1,R_1}$ into  $X_{T_1,R_1}$, and we can not conclude 
  directly that $\phi=F(\phi)$, in other words that $\phi$ satisfies~\eqref{kg}. By the 
  global well-posedness result, there exists  a unique  $\psi=F(\psi)$ for 
  $t \in [0,T_1]$. Let us prove that $\phi \equiv \psi $ on   $[0,T_1]$. Observe 
  that we have $\psi=F^k(\psi)$. To conclude the proof, by uniqueness of the 
  fixed point of $F^k$ in  $X_{T_1,R_1}$, it is enough to check that 
  $\psi \in X_{T_1,R_1}$. By~\eqref{boundEnergy}, 
   $\|\psi\|_{T_1}\leq C_0(T,L,\|\psi\|_0)\leq R_1$, hence the result.
  
   Finally the bound~\eqref{picard} directly follows from the Picard iteration procedure, since 
   \[\widetilde{Y}_{k(p+1),t,t_0}^u(\psi_0,\psi_1) =
   F^k\big(\widetilde{Y}_{kp,t,t_0}^u(\psi_0,\psi_1)\big).\]
\end{proof}

Recall that   $\|\psi\|_T= 
\| \psi \|_{L^{\infty}_TH^1}+\| \partial_t\psi \|_{L^{\infty}_TL^2}$.
 
 \begin{lem}\label{lemk}
 Let $0<T_1\leq T$. For $0\leq t\leq T$, set $L(t)=\int_0^t|u(s)|ds$ and $L=L(T)$. Then 
 there exists a constant $C>0$ such that for all $k\geq 0$ and 
 $0 \leq t +t_0\leq T$, there exist polynomials $C_k$, $P_k$ and $Q_k$ such that 
 \begin{equation}\label{bornek}
 \|F^k(\psi)\|_{t} \leq C_k(L, \|\psi\|_0)+\frac{\big(CL(t+t_0)\big)^{k}}{k!}  
 \|\psi\|_{T_1}+T_1P_k(T_1,L, \|\psi\|_{T_1})
 \end{equation}
 and 
  \begin{equation}\label{contratk}
 \|F^k(\psi)-F^k(\phi)\|_{t} \leq  \Big[\frac{\big(CL(t+t_0)\big)^{k}}{k!} + 
 T_1Q_k(T_1,L, \|\psi\|_{T_1},\|\phi\|_{T_1} )        
 \Big] \|\psi-\phi\|_{T_1}.
 \end{equation}
 \end{lem}

\begin{proof}
Let us prove~\eqref{bornek} by induction. For $k=0$ the result holds true. Let $k \geq 0$ so
that we have~\eqref{bornek}. As in~\eqref{ineq37}, we get 
   \begin{equation}\label{k+1}
 \| F^{k+1}(\psi)\|_t \leq   2(\|\widetilde\psi_0\|_{H^1}+
 \|\widetilde\psi_1\|_{L^2}) +c\big\| B \big\|_{L^{\infty}}  \big(\int_0^t  |u(s+t_0)| 
 \big\| F^{k}(\psi)\big\|_{s}ds \big)
 +cT_1 \|  F^{k}(\psi) \|^3_{T_1},
 \end{equation}
 where $c>0$ is a universal constant. Moreover, by~\eqref{boundEnergy}, 
 \[ \|\widetilde\psi_0\|_{H^1}+\|\widetilde\psi_1\|_{L^2}  \leq D(L ,\|\psi\|_0).
 \]
 Next, by~\eqref{bornek}
  \begin{multline*}
 \int_0^t  |u(s+t_0)| \big\| F^{k}(\psi)\big\|_{s}ds \leq \\
    \begin{aligned}
& \leq    C_k(L, \|\psi\|_0)  L+ \|\psi\|_{T_1}\int_0^t  |u(s+t_0)|   
\frac{\big(CL(s+t_0)\big)^{k}}{k!}   ds +
T_1LP_k(T_1,L, \|\psi\|_{T_1}) \\
& \leq  C_k(L, \|\psi\|_0) L+   C^k\frac{\big(L(t+t_0)\big)^{k+1}}{(k+1)!}      \|\psi\|_{T_1}
+T_1LP_k(T_1,L, \|\psi\|_{T_1}).
    \end{aligned}
 \end{multline*}
 The term  $\|  F^{k}(\psi) \|^3_{T_1}$ is directly controlled  
 by~\eqref{bornek}. Now  we make the choice $C=c\big\| B \big\|_{L^{\infty}}$, 
 and, thanks  to~\eqref{k+1} we get~\eqref{bornek} for $k+1$.
 
 The proof of~\eqref{contratk} is similar and omitted.
\end{proof}\medskip

As in the abstract result, a major ingredient of the proof is a  Dyson expansion 
of the form~\eqref{Dyson}. However, since the nonlinearity is stronger than in 
our abstract result, the expansion only holds for finite times. Set

\[\widetilde{Z}_{p,t,t_0}^u(\psi_0,\psi_1) := \widetilde{Y}_{k(p+1),t,t_0}^u(\psi_0,\psi_1) -\widetilde{Y}_{kp,t,t_0}^u(\psi_0,\psi_1),\]
 where $k\geq 0$ is given by the proof of Proposition~\ref{GWP-KG}.

 \begin{proposition}\label{PRO_Dyson_KG}
Let $T>0$ and $u \in L^1([0,T],\R)$ such that $\int_0^T|u(s)|ds\leq L$. Consider 
$\tau=\tau(\|\psi_0\|_{H^1}, \|\psi_1\|_{L^2}, L,T )>0$ given by 
Proposition~\ref{GWP-KG} $(\ref{3.2.iii})$. Then for all 
$t \in [0,\tau]$
\begin{equation*} 
 \Phi^u(t+t_0)\big(\psi_0, \psi_1\big) = 
\Big(   \sum_{j=0}^\infty \widetilde{Z}_{j,t,t_0}^u(\psi_0,\psi_1 )   ,    \sum_{j=0}^\infty \partial_t  \widetilde{Z}_{j,t,t_0}^u(\psi_0,\psi_1 )     \Big).
\end{equation*}
\end{proposition}

\begin{proof}
This result is a direct consequence of~\eqref{picard}.
\end{proof}

  \subsubsection{Proof of the compactness result}
  \label{SEC_compactness_argument_KG}
We now proceed to the end of the proof of Theorem~\ref{THM_BMS_KG}. For every 
$(\widetilde{\psi}_0,\widetilde{\psi}_1)$ in 
$H^1(\mathcal{M})\times L^2(\mathcal{M})$, we define 
the attainable set from $(\widetilde{\psi}_0,\widetilde{\psi}_1)$ in time less 
than~$T$ with controls whose $L^1$ norm is less than $L$:
\[
\mathcal{V}^{T,L}(\widetilde{\psi}_0,\widetilde{\psi}_1)= 
\big\{ \Phi^u(t) (\widetilde{\psi}_0,\widetilde{\psi}_1) \;| \; 
u \in L^1([0,T],\mathbf{R}), \|u\|_{L^1([0,T],\mathbf{R})} \leq L, 
0 \leq t \leq T \big\}.
\]

  \begin{proposition}\label{PRO_compctness_KG_local}
  For every $(\widetilde{\psi}_0,\widetilde{\psi}_1)$ in $H^1(\mathcal{M}) \times L^2(\mathcal{M})$, for every $L>0$, for every $T\leq \tau$ (defined in 
  Proposition~\ref{GWP-KG} $(\ref{3.2.iii})$), 
  $\mathcal{V}^{T,L}(\widetilde{\psi}_0,\widetilde{\psi}_1)$ is contained in a compact set  of $H^1(\mathcal{M}) \times L^2(\mathcal{M})$. 
  \end{proposition}
\begin{proof}
The proof of Proposition~\ref{PRO_compctness_KG_local} proceeds in exactly the same way as the proof 
of Theorem~\ref{THM_BMS_abstract}, using the Dyson expansion 
(Proposition~\ref{PRO_Dyson_KG}) and the fact that the mappings 
\[
\begin{array}{llcl}
G_1:&[0,T]\times L^2(\mathcal{M}) &\longrightarrow & H^1(\mathcal{M})\\
&(s,\phi) & \longmapsto &  \dis  
\frac{\sin((T-s) \sqrt{-\Delta+m} )}{ \sqrt{-\Delta+m}}B \phi
\end{array}
\]
and 
\[
\begin{array}{llcl}
G_2:&[0,T]\times H^1(\mathcal{M}) &\longrightarrow & H^1(\mathcal{M})\\
&(s,\psi) & \longmapsto &  \dis  
\frac{\sin((T-s) \sqrt{-\Delta+m} )}{ \sqrt{-\Delta+m}}\psi^3
\end{array}
\]
are continuous.
\end{proof} 

\begin{proposition}\label{PRO_tau_uniforme}
For every $(\psi_0,\psi_1)$ in $H^1(\mathcal{M}) \times L^2(\mathcal{M})$, and for every $L,T>0$, there exists $\tau^\ast >0$ such that, for every 
$(\widetilde{\psi}_0, \widetilde{\psi}_1)$ in the topological closure of
$\mathcal{V}^{T,L}({\psi}_0,{\psi}_1)$, 
the time $\tau$ given in Proposition~\ref{GWP-KG} $(\ref{3.2.iii})$ 
satisfies $\tau > \tau^\ast$.   
\end{proposition} 
\begin{proof}
The time $\tau$ appearing in Proposition~\ref{GWP-KG} $(\ref{3.2.iii})$ is the 
time $\tau$ for which the Dyson expansion (Proposition~\ref{PRO_Dyson_KG}) is 
valid. As proved in Proposition~\ref{GWP-KG}, this time depends on the norm of $
\psi_0$ and~$\psi_1$ (not on $\psi_0$ and $\psi_1$ themselves). The conclusion follows 
from the energy bound~\eqref{boundEnergy}.  
\end{proof}

\begin{proposition}\label{PRO_compactness_KG_T}
For every $T,L>0$ and $(\psi_0,\psi_1)$ in 
$H^1(\mathcal{M}) \times L^2(\mathcal{M})$, the set 
$ \mathcal{V}^{T,L}({\psi}_0,{\psi}_1)$ is relatively compact  in $H^1(\mathcal{M}) \times L^2(\mathcal{M})$. 
\end{proposition}
\begin{proof} 
In the following, for every real function $u:\mathbf{R}\to \mathbf{R}$ and every 
interval $I=[a,b]$ of $\mathbf{R}$, we define the function $R_I u$ by 
$R_I u (x)=u(a+x)$ for $x$ in $[0,b-a]$ and $R_I u (x)=0$ else.

Let $\tau^\ast$ be as defined in Proposition~\ref{PRO_tau_uniforme}. We proceed by induction on $p$ in $\mathbf{N}$ to prove 
Proposition~\ref{PRO_compactness_KG_T} for $T\leq p\tau^\ast$.

For $p=1$, this is just Proposition~\ref{PRO_compctness_KG_local}.

Assume the result holds for  $p\geq 1$.    
Let $T$ be in $(p\tau^\ast, (p+1) \tau^\ast]$ and 
$(A_n)_{n\in \mathbf{N}}= \big(\Phi^{u_n}(t_n)(\psi_0,\psi_1)\big)_{n\in 
\mathbf{N}}$ be a sequence in $ \mathcal{V}^{T,L}({\psi}_0,{\psi}_1)$. We aim 
to find a convergent subsequence of $(A_n)_{n\in \mathbf{N}}$, which will prove 
the relative compactness of $ \mathcal{V}^{T,L}({\psi}_0,{\psi}_1)$. 

By the induction hypothesis, the set
$\mathcal{V}^{p\tau^\ast,L}({\psi}_0,{\psi}_1)$ is relatively compact, hence up 
to extraction of a subsequence, one may assume that the sequence $\big(\Phi^{u_n}(p\tau^\ast)(\psi_0,
\psi_1)\big)_{n\in \mathbf{N}}$ converges to some limit $A^\infty_{p\tau^\ast}$. By 
Proposition~\ref{PRO_tau_uniforme}, $\tau(A^\infty_{p\tau^\ast},L)>\tau^\ast$. Hence, by 
Proposition~\ref{PRO_compctness_KG_local}, the set $ \mathcal{V}^{\tau^\ast,L}
(A^\infty_{p\tau^\ast})$ is relatively compact and, up to extraction of a subsequence, one may 
assume that the sequence $\big(\Phi^{R_{[p\tau^\ast, t_n]}u_n}(t_n-\tau^\ast)
(A^\infty_{\tau^\ast})\big)_{n\in \mathbf{N}}$ converges to some limit 
$A^\infty_{T_\infty}$. 
By continuity of $\Phi^u(t)(\cdot, \cdot)$, 
the sequence $\big(\Phi^{u_n}(t_n)(\psi_0,\psi_1)\big)_{n\in 
\mathbf{N}}$ also converges to $A^\infty_{T_\infty}$, and that 
concludes the proof of Proposition~\ref{PRO_compactness_KG_T}.
\end{proof}

\begin{proof}[Proof of Theorem~\ref{THM_BMS_KG}:]  
It remains to prove the last statement of Theorem~\ref{THM_BMS_KG}. This follows from 
Proposition~\ref{PRO_compactness_KG_T} by noticing that
$\displaystyle{ \bigcup_{t \in \R} \bigcup_{u \in L^1} 
\big\{\Phi^u(t) (\psi_0,\psi_1)\big\} \subset \bigcup_{\ell \in \mathbf{N}} 
\bigcup_{n \in \mathbf{N}}  \mathcal{V}^{n,\ell}({\psi}_0,{\psi}_1)}$.
\end{proof}

   %%%%%%%%%%%%%%%%%%%%%%%%%%%%%%%%%%%%%%%%%%%%%%%%%%%%%%%%%%%%%

 \appendix
 
\section{Sobolev spaces} \label{Appendix}

The aim of this Appendix is to recall the classical  Sobolev embedding theorem, which is instrumental in the proof of Proposition~\ref{GWP-KG}.  For more details, the reader may refer to the classical reference \cite[Theorem 5.4, statements (3) and (4)]{Adams}.
\subsection{Definition}

Let $\mathcal{M}$ be an open subset of $\mathbf{R}^n$ or 
a compact Riemannian  manifold of dimension~$n$. 
For every $k$ in $\mathbf{N}$ and every $p$ in $[1,+\infty]$, the  Sobolev 
space $W^{k,p}(\mathcal{M})$ is defined as the set of functions from $\mathcal{M}$ to 
$\mathbf{R}$ whose partial derivatives up to order $k$ belongs to
 $L^p(\mathcal{M})$, that is:
\[
W^{k,p}(\mathcal{M})=\big\{\psi \in L^p(\mathcal{M})\; |\; D^\alpha \psi \in L^p(\mathcal{M}), \quad \forall \, |\alpha|\leq k\big\}.
\]
When endowed with the norm $\|\psi \|_{W^{k,p}(\mathcal{M})}=
\sum_{|\alpha|\leq p}\|D^\alpha 
\psi \|_{L^p}$, $W^{k,p}(\mathcal{M})$ turns into a Banach space.

In the case where $p=2$, $W^{k,2}(\mathcal{M})$ turns into a Hilbert space and is 
usually denoted by $H^2(\mathcal{M})$.

\subsection{Sobolev embedding theorem}\label{SEC_Sobolev_Embedding}

For every integers $k,\ell$ and every real numbers $p,q$ such that $k>\ell$, 
 $(k-\ell)p<n$, and $1\leq p < q \leq np/(n-(k-\ell)p) \leq +\infty $, 
\[ W^{k,p}(\mathcal{M}) \subset W^{\ell,q}(\mathcal{M}) \]
and the embedding is continuous. In particular, 
there exists $C_{Sob}(p,q,k,n)>0$ such that 
\[ \|\psi\|_{W^{\ell,q}(\mathcal{M})} \leq C_{Sob}(p,q,k,n) 
\|\psi\|_{W^{k,p}(\mathcal{M})}.\]

In particular, if $k=1$ and $\ell=0$, one gets
\begin{proposition}[Sobolev embedding]\label{PRO_Sobolev_embedding}
If $1/p^\ast = 1/p-1/n$ then $W^{1,p}(\mathcal{M}) \subset L^{p^\ast} (\mathcal{M})$.
\end{proposition}
%\subsection{Kondrachov embedding theorem}

\section*{Acknowledgments}

T. Chambrion is supported by the grant ``QUACO'' ANR-17-CE40-0007-01.
 L. Thomann is supported by the grants ``BEKAM''  ANR-15-CE40-0001 
 and ``ISDEEC'' ANR-16-CE40-0013.   
   
\bibliographystyle{alpha}

\Addresses

\end{document}